\documentclass[12pt,twoside]{amsart}
\usepackage{amssymb}
\usepackage{amscd}
\usepackage{xypic}

\title[On log canonical singularities]{On isolated log canonical singularities with 
index one}
\dedicatory{Dedicated to Professor Shihoko Ishii on the 
occasion of her sixtieth birthday}
\author{Osamu Fujino} 

\date{2011/11/11, version 1.52}


\address{Department of Mathematics, Faculty of Science, 
Kyoto University, Kyoto 606-8502, Japan}
\email{fujino@math.kyoto-u.ac.jp}

\newcommand{\Supp}[0]{{\operatorname{Supp}}}
\newcommand{\Exc}[0]{{\operatorname{Exc}}}
\newcommand{\Gr}[0]{{\operatorname{Gr}}}

\newtheorem{thm}{Theorem}[section]
\newtheorem{lem}[thm]{Lemma}
\newtheorem{cor}[thm]{Corollary}
\newtheorem{prop}[thm]{Proposition}

\theoremstyle{definition}
\newtheorem{step}{Step}
\newtheorem{ex}[thm]{Example}
\newtheorem{defn}[thm]{Definition}
\newtheorem{rem}[thm]{Remark}
\newtheorem*{ack}{Acknowledgments}      
\newtheorem*{notation}{Notation}         
\newtheorem{say}[thm]{}
\begin{document}
\bibliographystyle{amsalpha+}

\maketitle

\begin{abstract}
We give a method to investigate 
isolated log canonical singularities with index one which are not log terminal. 
Our method depends on the minimal model program. 
One of the main purposes is to show that our invariant coincides with 
Ishii's Hodge theoretic invariant. 
\end{abstract}

\tableofcontents

\section{Introduction} 

Let $P\in X$ be an $n$-dimensional isolated log canonical singularity 
with index one which is not log terminal. 
Let $f:Y\to X$ be a projective resolution such that $f$ is an isomorphism 
outside $P$ and that $\Supp f^{-1}(P)$ is a simple normal crossing 
divisor on $Y$. 
Then we can write 
$$
K_Y=f^*K_X+F-E 
$$
where $E$ and $F$ are effective Cartier divisors 
and have no common irreducible components. 
The divisor $E$ is sometimes called the 
{\em{essential divisor}} for $f$ (see \cite[Definition 7.4.3]{ishii-book} 
and \cite[Definition 2.5]{ishii-q}). 

In \cite[Propositions 1.4 and 3.7]{ishii}, 
Shihoko Ishii proves  
$$
R^{n-1}f_*\mathcal O_Y\simeq H^{n-1}(E, \mathcal O_E)\simeq \mathbb C. 
$$ 
For details, see \cite[Propositions 5.3.11, 5.3.12, 
7.1.13, 7.4.4, and Theorem 7.1.17]{ishii-book}. 
In this paper, we prove that 
$$
R^if_*\mathcal O_Y\simeq H^i(E, \mathcal O_E)
$$ 
for {\em{every}} $i>0$ (cf.~Proposition \ref{46}) 
and that 
$$
R^{n-1}f_*\mathcal O_Y\simeq \mathbb C(P) 
$$ 
(cf.~Remark \ref{4646}). 
Our proof depends on the minimal model theory and is different from 
Ishii's. 

By Shihoko Ishii, 
the singularity $P\in X$ is said to be of type $(0, i)$ if 
\begin{align*}
\Gr ^W_kH^{n-1}(E, \mathcal O_E)
=\begin{cases}
\mathbb C &\text{if}\quad  k=i\\ 
0 &\text{otherwise}
\end{cases}
\end{align*}
where $W$ is the weight filtration of the mixed Hodge 
structure on $H^{n-1}(E, \mathbb C)$. Note that 
$E$ is a projective connected simple normal crossing 
variety. 
Therefore, we have 
\begin{align*}
&\Gr ^W_kH^{n-1}(E, \mathcal O_{E})\\
&\simeq \Gr ^W_k \Gr ^0 _FH^{n-1}(E, \mathbb C)\\ 
&\simeq \Gr ^0_F\Gr ^W _k H^{n-1} (E, \mathbb C)
\end{align*}
where $F$ is the Hodge filtration. 
We also note that the type of $P\in X$ is independent of the choice of 
a resolution $f:Y\to X$ by \cite[Proposition 4.2]{ishii} 
(see also \cite[Proposition 7.4.6]{ishii-book}). 

On the other hand, we define 
$\mu(P\in X)$ by 
$$
\mu=\mu(P\in X)=\min \{ \dim W\, |\, W \ {\text{is a stratum of}}\  E\}
$$ 
(see \cite[Definition 4.12]{fujino2}). 
We prove that $P\in X$ is of type $(0, \mu)$, that is, 
Ishii's Hodge theoretic invariant coincides with our 
invariant $\mu$ (cf.~Theorem \ref{new2}). 
It was first obtained by Shihoko Ishii in \cite{ishii-letter}. 

By our method based on the minimal model program, 
we can prove the following properties of $E$. 
Let $E=\sum _i E_i$ be the irreducible decomposition. 
Then $\sum _{i\ne i_0}E_i|_{E_{i_0}}$ has at most two connected 
components for every irreducible component $E_{i_0}$ of $E$ (cf.~Remark \ref{rem48}). 
Let $W_1$ and $W_2$ be any two minimal strata of $E$. 
Then $W_1$ is birationally equivalent to $W_2$ (cf.~\ref{4111} and Remark \ref{rem48}). 
These results seem to be out of reach by the Hodge theoretic 
method.  

Let $\Gamma$ be the dual complex of $E$ and let $|\Gamma|$ be the topological 
realization of $\Gamma$. Then the dimension of $|\Gamma|$ is $n-1-\mu$ by the definition 
of $\mu$. 

From now on, we assume that $\mu(P\in X)=0$. In this case, 
we can prove that 
$$
H^i(E, \mathcal O_E)\simeq H^i(|\Gamma|, \mathbb C)
$$ 
for every $i$. 
Therefore, $P\in X$ is Cohen--Macaulay, equivalently, 
Gorenstein,  if and only if 
$$
H^i(|\Gamma|, \mathbb C)
=\begin{cases}
\mathbb C &\text{if}\quad  i=0, \ n-1, \\ 
0 &\text{otherwise.}
\end{cases}
$$
It is Theorem \ref{4-10}. 

Anyway, by this paper, our approach based on the minimal model program 
(cf.~\cite{fujino2}) 
becomes compatible with Ishii's Hodge theoretic method in \cite{ishii}, \cite{ishii-book}, 
and \cite{ishii-q}. Our approach is more geometric than Ishii's. 
From our point of view, the main result of \cite{iw} becomes almost obvious. 
We note that we do not use the notion of {\em{Du Bois singularities}}, which is 
one of the main ingredients of Ishii's Hodge theoretic 
approach. 

We summarize the contents of this paper. 
Section \ref{sec-a} is a preliminary section. 
In Section \ref{sec21}, we give a criterion of Cohen--Macaulayness. 
In Section \ref{sec22}, we investigate basic properties of 
dlt pairs. In Section \ref{sec23}, we explain the 
notion of {\em{dlt blow-ups}}, which is very useful in the 
subsequent sections. 
Section \ref{sec3} is devoted to the study of dlt pairs with 
torsion log canonical divisor. 
In Section \ref{sec2}, we investigate isolated lc singularities with 
index one which are not log terminal. 
In Section \ref{sec5}, we prove that our invariant $\mu$ coincides 
with Ishii's Hodge theoretic invariant. 
The main result (cf.~Theorem \ref{new1}) in Section \ref{sec5} 
can be applied to special fibers of semi-stable 
minimal models for varieties with trivial canonical 
divisor (cf.~\cite{fujino-ss}). 

\begin{notation}
Let $X$ be a normal variety and let $B$ be an effective 
$\mathbb Q$-divisor such that 
$K_X+B$ is $\mathbb Q$-Cartier. 
Then we can define the {\em{discrepancy}} 
$a(E, X, B)\in \mathbb Q$ for 
every prime divisor $E$ {\em{over}} $X$. 
If $a(E, X, B)\geq -1$ (resp.~$>-1$) for 
every $E$, then $(X, B)$ is called {\em{log canonical}} 
(resp.~{\em{kawamata log terminal}}). 
We sometimes abbreviate log canonical (resp.~kawamata log terminal) 
to {\em{lc}} (resp.~{\em{klt}}). 
When $(X, 0)$ is klt, we simply say that 
$X$ is {\em{log terminal}} ({\em{lt}}, for short). 

Assume that $(X, B)$ is log canonical. 
If $E$ is a prime divisor over $X$ 
such that 
$a(E, X, B)=-1$, 
then $c_X(E)$ is called a {\em{log canonical center}} 
({\em{lc center}}, for short) 
of $(X, B)$, where $c_X(E)$ is 
the closure of the 
image 
of $E$ on $X$. 

Let $T$ be a simple normal crossing variety (cf.~Definition \ref{def26}) and let 
$T=\sum _{i\in I} T_i$ be the irreducible decomposition. Then 
a {\em{stratum}} of $T$ is an irreducible component 
of $T_{i_1}\cap \cdots \cap T_{i_k}$ for some 
$\{i_1, \cdots, i_k\}\subset I$. 

Let $r$ be a rational number. 
The integral part $\llcorner r\lrcorner$ is 
the largest integer $\leq r$ and 
the fractional part $\{r\}$ is defined by $r-\llcorner r\lrcorner$. 
We put $\ulcorner r\urcorner =-\llcorner -r \lrcorner$ and 
call it the round-up of $r$. 
Let $D=\sum _{i=1}^r d_i D_i$ be a $\mathbb Q$-divisor 
where $D_i$ is a prime divisor for every $i$ and $D_i\ne D_j$ for 
$i\ne j$. 
We put $\llcorner D\lrcorner =
\sum \llcorner d_i\lrcorner D_i$, 
$\ulcorner D\urcorner =
\sum \ulcorner d_i\urcorner D_i$, 
$\{D\} =
\sum \{d_i\} D_i$, and 
$D^{=1}=\sum _{d_i=1}D_i$. 
\end{notation}

\begin{ack}
The first version of this paper was written in Nagoya in 2007. 
The author was partially supported by the Grant-in-Aid for Young Scientists 
(A) $\sharp$17684001 from JSPS. 
In 2011, he revised and expanded it in Kyoto. 
He was partially supported by the Grant-in-Aid for Young Scientists 
(A) $\sharp$20684001 from JSPS. 
He was also supported by the Inamori Foundation. 
He would like to thank Professor Shihoko Ishii very much 
for useful comments and her suggestive 
talks in Kyoto in the late 1990's. 
He thanks Professors Kenji Matsuki and Masataka Tomari for 
useful comments. 
He also thanks the referee for careful reading and many useful comments. 
Finally, he thanks Professor Shunsuke Takagi for useful comments, discussions, 
and questions. 
This paper is a supplement 
to \cite{fujino2}, \cite{ishii-q}, and \cite[Chapter 7]{ishii-book}.  
\end{ack}

In this paper, we will work over $\mathbb C$, the complex number 
field. We will freely make use of the standard notation and 
definition in \cite{km}. 

\section{Preliminaries}\label{sec-a}
In this section, we prove some preliminary results. 

\subsection{A criterion of Cohen--Macaulayness}\label{sec21} 
The main purpose of this subsection 
is to prove Corollary \ref{333}, which seems to be 
well known to experts. 
Here, we give a global proof based on the Kawamata--Viehweg vanishing 
theorem for the reader's convenience. 
See also the arguments in 
\cite[4.3.1]{fujino-book}.  

\begin{lem}\label{lem1}
Let $X$ be a normal variety with an isolated 
singularity $P\in X$. 
Let $f:Y\to X$ be any resolution. 
If $X$ is Cohen--Macaulay, 
then $R^if_*\mathcal O_Y=0$ for 
$0<i<n-1$, where 
$n=\dim X$. 
\end{lem}
\begin{proof}
Without loss of generality, we may assume 
that $X$ is projective. We consider the following 
spectral sequence 
$$
E^{p,q}_2=H^p(X, R^qf_*\mathcal O_Y\otimes 
L^{-1})\Rightarrow H^{p+q}(Y, f^*L^{-1})
$$ 
for a sufficiently ample line bundle $L$ on $X$. 
By the Kawamata--Viehweg vanishing theorem, 
$H^{p+q}(Y, f^*L^{-1})=0$ 
for $p+q<n$. On the other 
hand, $E^{p,0}_2=H^p(X, L^{-1})=0$ 
for $p<n$ since $X$ is Cohen--Macaulay. By using 
the exact sequence 
$$0\to E^{1,0}_2\to E^1\to 
E^{0,1}_2\to E^{2,0}_2\to E^2\to \cdots,$$ 
we obtain 
$E^{0,1}_2\simeq E^{2,0}_2=0$ when $n\geq 3$. 
This implies $R^1f_*\mathcal O_Y=0$. 
We note that 
$\Supp R^if_*\mathcal O_Y\subset \{P\}$ for 
every $i>0$. 
Inductively, we obtain 
$R^if_*\mathcal O_Y
\simeq H^0(X, R^if_*\mathcal O_Y\otimes 
L^{-1})=E^{0,i}_2\simeq E^{0, i}_{\infty}=0$ for 
$0<i<n-1$. 
\end{proof}

\begin{lem}\label{lem2}
Let $X$ be a normal projective $n$-fold 
and let $f:Y\to X$ 
be a resolution. Assume that 
$R^if_*\mathcal O_Y=0$ for $0<i<n-1$. Then 
$X$ is Cohen--Macaulay. 
\end{lem}
\begin{proof}
It is sufficient to prove 
$H^i(X, L^{-1})=0$ for any 
ample line bundle $L$ on $X$ for all 
$i<n$ (see \cite[Corollary 5.72]{km}). 
We consider the spectral sequence 
$$
E^{p,q}_2=H^p(X, R^qf_*\mathcal O_Y\otimes 
L^{-1})\Rightarrow H^{p+q}(Y, f^*L^{-1}). 
$$ 
As before, $H^{p+q}(Y, f^*L^{-1})=0$ for $p+q<n$ by 
the Kawamata--Viehweg vanishing theorem. 
By the exact sequence 
$$0\to E^{1,0}_2\to E^1\to 
E^{0,1}_2\to E^{2,0}_2\to E^2\to \cdots,$$ 
we obtain $H^1(X, L^{-1})=0$ and $H^2(X, L^{-1})=0$ if 
$n\geq 3$. 
Inductively, we can check that 
$H^i(X, L^{-1})=E^{i, 0}_2\simeq E^{i, 0}_{\infty}=0$ for 
$i<n$. 
We finish the proof. 
\end{proof}

Combining the above two lemmas, we obtain 
the next corollary. 

\begin{cor}\label{333}
Let $P\in X$ be a normal isolated 
singularity and let $f:Y\to X$ be a resolution. 
Then $X$ is Cohen--Macaulay if and only if 
$R^if_*\mathcal O_Y=0$ for $0<i<n-1$, 
where $n=\dim X$. 
\end{cor}
\begin{proof}
We shrink $X$ and assume that 
$X$ is affine. 
Then we compactify $X$ and may assume 
that $X$ is projective. 
Therefore, we can apply Lemmas \ref{lem1} and 
\ref{lem2}. 
\end{proof}

\subsection{Basic properties of dlt pairs}\label{sec22} 
In this subsection, we prove supplementary results on dlt pairs. 
For the definition of dlt pairs, see \cite[Definition 2.37, Theorem 2.44]{km}. 
See also \cite{fujino-what} for details of singularities of pairs.  

The following proposition generalizes \cite[17.5 Corollary]{fa}, 
where it was only proved that $S$ is semi-normal and $S_2$. 
In the subsequent sections, we will use the arguments 
in the proof of Proposition \ref{31}. 

\begin{prop}[{cf.~\cite[Theorem 4.4]{fujino-book}}]\label{31} 
Let $(X, \Delta)$ be a dlt pair and 
let $\llcorner \Delta\lrcorner =: 
S=S_1+\cdots +S_k$ be the irreducible 
decomposition. 
We put $T=S_1+\cdots +S_l$ for $1\leq l \leq k$. 
Then $T$ is semi-normal, Cohen--Macaulay, and has only 
Du Bois singularities. 
\end{prop}
\begin{proof}
We put $B=\{\Delta\}$. 
Let $f:Y\to X$ be a resolution such that 
$K_Y+S'+B'=f^*(K_X+S+B)+E$ with the 
following properties:~
(i) $S'$ (resp.~$B'$) is 
the strict transform of $S$ (resp.~$B$), 
(ii) $\Supp (S'+B')\cup \Exc (f)$ and $\Exc(f)$ are 
simple normal crossing divisors on $Y$, 
(iii) $f$ is an isomorphism over the 
generic point of every lc center of 
$(X, S+B)$, and (iv) $\ulcorner E\urcorner \geq 0$. 
We write $S=T+U$. Let $T'$ (resp.~$U'$) 
be the strict transform of $T$ (resp.~$U$) 
on $Y$. We consider the following 
short exact sequence 
$$0\to \mathcal O_Y(-T'+\ulcorner E\urcorner)\to 
\mathcal O_Y(\ulcorner E\urcorner)\to \mathcal O_{T'}(\ulcorner 
E|_{T'}\urcorner)\to 0. $$ 
Since 
$-T'+E\sim _{\mathbb Q, f}K_Y+U'+B'$ and 
$E\sim _{\mathbb Q, f}K_Y+S'+B'$, 
we have $-T'+\ulcorner E\urcorner \sim _{\mathbb Q, f}K_Y
+U'+B'+\{-E\}$ and 
$\ulcorner E\urcorner \sim _{\mathbb Q, f} K_Y+S'+B'+\{-E\}$. 
By the vanishing theorem of Reid--Fukuda type (see, for example, 
\cite[Lemma 4.10]{fujino-book}), 
$$R^if_*\mathcal O_Y(-T'+\ulcorner E\urcorner)=R^if_*\mathcal O_Y(\ulcorner 
E\urcorner)=0$$ for every $i>0$. 
Note that we used 
the assumption that $f$ is an isomorphism over the generic point of every lc center 
of $(X, S+B)$. 
Therefore, we have $$0\to 
f_*\mathcal O_Y(-T'+\ulcorner E\urcorner)\to 
\mathcal O_X\to f_*\mathcal O_{T'}(\ulcorner 
E|_{T'}\urcorner)\to 0$$ and 
$R^if_*\mathcal O_{T'}(\ulcorner E|_{T'}\urcorner)=0$ 
for all $i>0$. Note that $\ulcorner E\urcorner $ is effective 
and $f$-exceptional. 
Thus, $\mathcal O_T\simeq 
f_*\mathcal O_{T'}\simeq f_*\mathcal O_{T'} 
(\ulcorner E'|_{T'}\urcorner)$. 
Since $T'$ is a simple normal crossing divisor, 
$T$ is semi-normal. 
By the above vanishing result, 
we obtain $Rf_*\mathcal O_{T'}(\ulcorner 
E|_{T'}\urcorner)\simeq 
\mathcal O_T$ in the derived category. 
Therefore, the composition 
$\mathcal O_T\to R f_*\mathcal O_{T'}\to 
Rf_*\mathcal O_{T'}(\ulcorner E|_{T'}\urcorner)\simeq 
\mathcal O_T$ is 
a quasi-isomorphism. 
Apply $R\mathcal Hom_T (\underline{\ \ \ } , \omega^{\bullet}_T)$ 
to the quasi-isomorphism 
$\mathcal O_T\to Rf_*\mathcal O_{T'}\to \mathcal O_T$. 
Then the composition 
$\omega^{\bullet}_T\to R f_*\omega^{\bullet}_{T'} 
\to \omega^{\bullet}_T$ is a quasi-isomorphism 
by the Grothendieck duality. 
By the vanishing theorem (see, for 
example, \cite[Lemma 2.33]{fujino-book}), 
$R^if_*\omega_{T'}=0$ for 
$i>0$. Hence, 
$h^i(\omega^{\bullet}_T)\subseteq R^if_*\omega^{\bullet}_{T'} 
\simeq R^{i+d}f_*\omega_{T'}$, where 
$d=\dim T=\dim T'$. Therefore, $h^i(\omega^{\bullet}_T)=0$ for $i>-d$. 
Thus, $T$ is Cohen--Macaulay. 
This argument is the same as the proof 
of Theorem 1 in \cite{kovacs2}. Since 
$T'$ is a simple normal crossing divisor, $T'$ has only Du Bois 
singularities. 
The quasi-isomorphism $\mathcal O_T\to 
Rf_*\mathcal O_{T'}\to \mathcal O_T$ implies that $T$ has 
only Du Bois 
singularities (cf.~\cite[Corollary 2.4]{kovacs1}). 
Since $T'$ is a simple normal crossing divisor on $Y$ and $\omega_{T'}$ is an invertible 
sheaf on $T'$, every associated prime of $\omega_{T'}$ is the generic point of some irreducible 
component of $T'$. By $f$, 
every irreducible component of $T'$ is mapped birationally onto 
an irreducible component of $T$. Therefore, $f_*\omega_{T'}$ is torsion-free on $T$. 
Since the composition $\omega_T\to f_*\omega_{T'}\to \omega_T$ 
is an isomorphism, we obtain 
$f_*\omega_{T'}\simeq \omega_T$. 
It is because $f_*\omega_{T'}$ is torsion-free and 
$f_*\omega_{T'}$ is generically isomorphic to $\omega_T$. 
By the Grothendieck duality, 
$$Rf_*\mathcal O_{T'}\simeq 
R\mathcal Hom _T(Rf_*\omega^{\bullet}_{T'}, 
\omega^{\bullet}_{T})\simeq 
R\mathcal Hom _T(\omega^{\bullet}_T, 
\omega^{\bullet}_T)\simeq \mathcal O_T. $$ 
So, $R^if_*\mathcal O_{T'}=0$ for all $i>0$. 
\end{proof}

We obtain the following vanishing theorem 
in the proof of Proposition \ref{31}. 

\begin{cor}\label{35} 
Under the notation in the proof of {\em{Proposition \ref{31}}}, 
$R^if_*\mathcal O_{T'}=0$ for 
every $i>0$ and $f_*\mathcal O_{T'}\simeq 
\mathcal O_T$.  
\end{cor}

We close this subsection with a useful lemma for simple normal crossing varieties. 

\begin{defn}[Normal crossing and simple normal crossing varieties]\label{def26}
A variety $X$ has {\em{normal crossing singularities}} 
if, for every closed point $x\in X$, 
$$
\widehat{\mathcal O}_{X, x}\simeq \frac{\mathbb C[[x_0, \cdots, x_N]]}{(x_0\cdots x_k)}
$$ 
for some $0\leq k\leq N$, where 
$N=\dim X$. 
Furthermore, if each irreducible component of $X$ is smooth, 
$X$ is called a {\em{simple normal crossing}} variety. 
\end{defn}

\begin{lem}\label{lemA}
Let $f:V_1\to V_2$ be a birational morphism between projective 
simple normal crossing varieties. 
Assume that there is a Zariski open subset $U_1$ {\em{(}}resp.~$U_2${\em{)}} of $V_1$ 
{\em{(}}resp.~$V_2${\em{)}} such that 
$U_1$ {\em{(}}resp.~$U_2${\em{)}} contains the generic point of any stratum of $V_1$ 
{\em{(}}resp.~$V_2${\em{)}} and that 
$f$ induces an isomorphism between $U_1$ and $U_2$. 
Then $R^if_*\mathcal O_{V_1}=0$ for every $i>0$ and $f_*\mathcal O_{V_1}\simeq 
\mathcal O_{V_2}$. 
\end{lem}
\begin{proof}
We can write 
$$
K_{V_1}=f^*K_{V_2}+E
$$ 
such that $E$ is $f$-exceptional. 
We consider the following commutative diagram 
$$
\begin{CD}
V_1^{\nu}@>{f^\nu}>> V_2^{\nu}\\
@V{\nu_1}VV @VV{\nu_2}V \\
V_1@>>{f}> V_2
\end{CD}
$$ 
where 
$\nu_1: V_1^\nu\to V_1$ and $\nu_2: V_2^\nu \to V_2$ are the normalizations. 
We can write $K_{V_1^{\nu}}+\Theta_1=\nu_1^*K_{V_1}$ and 
$K_{V_2^{\nu}}+\Theta_2=\nu_2^*K_{V_2}$, where 
$\Theta_1$ and $\Theta_2$ are the {\em{conductor}} divisors. 
By pulling back $K_{V_1}=f^*K_{V_2}+E$ to $V_1^\nu$ by $\nu_1$, 
we have 
$$
K_{V_1^\nu}+\Theta_1=(f^\nu)^*(K_{V_2^\nu}+\Theta_2)+\nu_1^*E. 
$$ 
Note that $V_2^\nu$ is smooth and $\Theta_2$ is a reduced simple normal crossing divisor on $V_2^\nu$. 
By the assumption, $f^\nu$ is an isomorphism 
over the generic point of any lc center of the 
pair $(V_2^\nu, \Theta_2)$. 
Therefore, $\nu_1^*E$ is effective since $K_{V_2^\nu}+\Theta_2$ is Cartier. 
Thus, we obtain that $E$ is effective. 
We can easily check that $f$ has connected fibers by 
the assumptions. 
Since $V_2$ is semi-normal and 
satisfies Serre's $S_2$ condition, 
we have $\mathcal O_{V_2}\simeq f_*\mathcal O_{V_1}$ and $f_*\mathcal O_{V_1}
(K_{V_1})\simeq \mathcal O_{V_2}(K_{V_2})$. 
On the other hand, 
we obtain $R^if_*\mathcal O_{V_1}(K_{V_1})=0$ for 
every $i>0$ by \cite[Lemma 2.33]{fujino-book}. 
Therefore, $Rf_*\mathcal O_{V_1}(K_{V_1})\simeq \mathcal O_{V_2}(K_{V_2})$ in the derived 
category. 
Since $V_1$ and $V_2$ are Gorenstein, 
we have $Rf_*\mathcal O_{V_1}\simeq \mathcal O_{V_2}$ in the 
derived category by the Grothendieck duality (cf.~the proof of 
Proposition \ref{31}). 
\end{proof}

\subsection{Dlt blow-ups}\label{sec23}
\label{mmp-ne} 
Let us recall the notion of {\em{dlt blow-ups}}. 
Theorem \ref{dltblowup} was 
first obtained by Christopher 
Hacon (cf.~\cite[Section 10]{fujino-fundamental}). 
For a simplified proof, see \cite[Section 4]{fujino-ss}. 

\begin{thm}[Dlt blow-up]\label{dltblowup} 
Let $(X, \Delta)$ be a quasi-projective 
lc pair. 
Then we can construct a projective 
birational morphism $f:Y\to X$ such that 
$K_Y+\Delta_Y=f^*(K_X+\Delta)$ with 
the following properties. 
\begin{itemize}
\item[(a)] $(Y, \Delta_Y)$ is a $\mathbb Q$-factorial dlt pair. 
\item[(b)] $a(E, X, \Delta)=-1$ for every $f$-exceptional 
divisor $E$. 
\end{itemize}
When $(X, \Delta)$ is dlt, we can make $f$ small 
and an isomorphism over the generic point of every lc center of $(X, \Delta)$. 
\end{thm}

Note that 
Theorem \ref{dltblowup} was proved by the minimal model program with 
scaling (cf.~\cite{bchm}). 

As a corollary of Theorem \ref{dltblowup}, 
we obtain the following useful lemma.  

\begin{lem}\label{conj1} 
Let $P\in X$ be an isolated lc singularity  
with index one, where $X$ is quasi-projective. 
Then there exists a projective 
birational morphism $g:Z\to X$ such 
that $K_Z+D=g^*K_X$, $(Z, D)$ is a $\mathbb Q$-factorial 
dlt pair, $g$ is an isomorphism outside $P$, 
and $D$ is a reduced divisor on $Z$.  
\end{lem}

\begin{rem}
If $P\in X$ is $\mathbb Q$-factorial, then $f^{-1}(P)$ is 
a divisor. 
So, we have $\Supp D=f^{-1}(P)$. 
In general, we have only 
$\Supp D\subset f^{-1}(P)$. 
\end{rem}

For non-degenerate isolated hypersurface log canonical 
singularities, 
we can use the toric geometry to construct dlt blow-ups 
as in Lemma \ref{conj1} (see \cite[Section 6]{fs}). 

\section{Dlt pairs with torsion log canonical divisor}\label{sec3}

This section is a supplement to \cite[Section 2]{fujino} and 
\cite[Section 2]{fujino2}. 
We introduce a new invariant for dlt pairs with torsion 
log canonical divisor. 

\begin{defn}\label{new} 
Let $(X, D)$ be a projective dlt pair such that 
$K_X+D\sim_{\mathbb Q} 0$. We put 
$$\widetilde\mu=\widetilde \mu (X, D)=\min \{\,\dim W \,| \, 
W \ {\text{is an lc center of}} \ (X, D)\}. 
$$ 
It is related to the invariant $\mu$, which is defined 
in \cite{fujino2} and will play important roles in the 
subsequent sections. See \ref{4111} below.  
\end{defn}

\begin{rem}
By \cite[Theorem 1]{ckp} or \cite[Theorem 1.2]{gongyo-zero}, 
$K_X+D\equiv 0$ if and only if $K_X+D\sim _{\mathbb Q}0$.  
\end{rem}

As we pointed out in \cite{fg}, \cite[Section 2]{fujino} works 
in any dimension by using the minimal model program with scaling 
(cf.~\cite{bchm}). 
Therefore, we obtain the following proposition (cf.~\cite[Proposition 2.4]{fujino2}). 

\begin{prop}\label{216} 
Let $(X, D)$ be a projective dlt pair 
such that 
$K_X+D\sim_{\mathbb Q} 0$. Let $W$ be any minimal lc center 
of $(X, D)$. 
Then $\dim W=\widetilde \mu(X, D)$. Moreover, all the 
minimal lc centers of $(X, D)$ are birational 
each other and $\llcorner D\lrcorner$ has at most two connected components. 
\end{prop}

\begin{proof}[Sketch of the proof] 
By Theorem \ref{dltblowup}, we may assume that 
$X$ is $\mathbb Q$-factorial. The induction on dimension 
and \cite[Proposition 2.1]{fujino} implies 
the desired properties. 
More precisely, all the minimal lc centers are $B$-birational 
each other (cf.~\cite[Definition 1.5]{fujino}). 
Note that Proof of Claims in the proof of \cite[Lemma 4.9]{fujino} 
may help us understand this proposition.  
\end{proof}

The next lemma is new. We will use it in Section \ref{sec2}. 

\begin{lem}\label{13}
Let $(X, D)$ be an $n$-dimensional 
projective dlt pair such 
that $K_X+D\sim _{\mathbb Q}0$. Assume that $\llcorner 
D\lrcorner \ne 0$. 
Then there exists an irreducible component $D_0$ of $\llcorner D\lrcorner$ 
such that 
$h^i(X, \mathcal O_X)\leq h^i(D_0, \mathcal O_{D_0})$ 
for every $i$. 
\end{lem}
\begin{proof}
By using the dlt blow-up (cf.~Theorem \ref{dltblowup}), 
we can construct a small projective 
$\mathbb Q$-factorialization of $X$.  
So, by replacing $X$  with its $\mathbb Q$-factorialization, 
we may assume that 
$X$  is $\mathbb Q$-factorial. 
By the assumption, $K_X+D-\varepsilon\llcorner D\lrcorner$ 
is not pseudo-effective for $0<\varepsilon \ll 1$. 
Let $H$ be an effective ample $\mathbb Q$-divisor on $X$ such that 
$K_X+D-\varepsilon\llcorner D\lrcorner+ H$ is nef and 
klt. 
Apply the minimal model program on $K_X+D-\varepsilon \llcorner D\lrcorner$ 
with scaling of $H$. 
Then we obtain a sequence 
of divisorial contractions and flips: 
$$
X=X_0\dashrightarrow X_1\dashrightarrow \cdots \dashrightarrow 
X_k,  
$$ 
and an extremal Fano contraction $\varphi: X_k\to Z$ (cf.~\cite[Section 2]{fujino-ss}). 
By the construction, there is an irreducible 
component $D_0$ of $\llcorner D\lrcorner$ such that the strict transform $D'_0$ of $D_0$ 
on $X_k$ dominates $Z$. 
Since $X$ and $X_k$ have only rational singularities, 
we have $h^i(X, \mathcal O_X)=h^i(X_k, \mathcal O_{X_k})$ for every $i$. 
Since $R^i\varphi_*\mathcal O_{X_k}=0$ for 
every $i>0$, 
we have $h^i(X_k, \mathcal O_{X_k})=h^i(Z, \mathcal O_Z)$ for 
every $i$. 
Since $D_0$ and $Z$ have only rational 
singularities (cf.~\cite[Corollary 1.5]{fujino-application}), 
$h^i(Z, \mathcal O_Z)\leq h^i(D_0, \mathcal O_{D_0})$ for 
every $i$ (see, for example, \cite[Theorem 2.29]{ps}). Therefore, we have the desired 
inequality $h^i(X, \mathcal O_X)\leq h^i(D_0, \mathcal O_{D_0})$ 
for every $i$. 
\end{proof}

\begin{ex}
Let $X=\mathbb P^2$ and let $D$ be an elliptic 
curve on $X=\mathbb P^2$. 
Then $(X, D)$ is a projective dlt pair such that $K_X+D\sim 0$. 
In this case, $h^1(X, \mathcal O_X)=0<h^1(D, \mathcal O_D)=1$. 
\end{ex}

By combining the above results, 
we obtain the next proposition. 

\begin{prop}\label{15}
Let $(X, D)$ be a projective dlt pair such that 
$K_X+D\sim _{\mathbb Q}0$. 
We assume that $\widetilde \mu (X, D)=0$. 
Then $h^i(X, \mathcal O_X)=0$ for 
every $i>0$. Moreover, $X$ is rationally connected. 
\end{prop}

\begin{proof}
If $\dim X=1$, then the statement is trivial since $X\simeq \mathbb P^1$. 
From now on, we assume that $\dim X\geq 2$. 
Since $\widetilde \mu(X, D)=0$, we obtain that $(X, D)$ is not klt. 
Thus we know $\llcorner D\lrcorner\ne 0$. 
Let $D_0$ be any irreducible component of $\llcorner 
D\lrcorner$. By adjunction, we obtain 
$(K_X+D)|_{D_0}=K_{D_0}+B$ such that 
$(D_0, B)$ is dlt, $K_{D_0}+B\sim_{\mathbb Q}0$, and $\widetilde \mu(D_0, B)=0$ 
by Proposition 
\ref{216}. 
By the induction on dimension, we know that 
every irreducible component $D_0$ of $\llcorner D\lrcorner$ is rationally connected and 
$h^i(D_0, \mathcal O_{D_0})=0$ for every $i>0$. 
Thus, by Lemma \ref{13}, we have that $h^i(X, \mathcal O_X)=0$ for 
every $i>0$. 
In the proof of Lemma 
\ref{13}, $Z$ has only log terminal singularities 
by \cite[Corollary 4.5]{fujino-application}. Since 
$D_0$ is rationally connected, 
so is $Z$ by \cite[Corollary 1.5]{hm}. 
On the other hand, the general fiber of $\varphi:X_k\to Z$ is rationally 
connected (cf.~\cite[Theorem 1]{zhang} and \cite[Corollaries 1.3 and 1.5]{hm}). 
By \cite[Corollary 1.3]{ghs}, $X_k$ is 
rationally connected. 
Thus, $X$ is rationally connected by \cite[Corollary 1.5]{hm}.   
\end{proof}

By Proposition \ref{15}, we obtain a corollary:~Corollary \ref{3777}. 

\begin{cor}\label{3777} 
Let $(X, D)$ be a projective dlt pair 
such that 
$K_X+D\sim _{\mathbb Q} 0$. 
Let $f:Y\to X$ be any resolution such that $K_Y+D_Y=f^*(K_X+D)$ and 
that 
$\Supp D_Y$ is a simple normal crossing divisor on $Y$. 
Assume that $\widetilde \mu (X, D)=0$. 
Then every stratum of $D^{=1}_Y$ is rationally connected. 
Moreover, $h^i(W, \mathcal O_W)=0$ for every $i>0$ where 
$W$ is a stratum of $D^{=1}_Y$. 
\end{cor}
\begin{proof}
Let $W$ be a stratum of $D^{=1}_Y$. 
Let $\pi:Y'\to Y$ be a blow-up 
at $W$ and let $E_W$ be the exceptional 
divisor of $\pi$. 
Then it is sufficient to 
prove that $E_W$ is rationally connected and $h^i(E_W, \mathcal O_{E_W})=0$ for 
every $i>0$. Therefore, by replacing $Y$ with $Y'$, we may assume that 
$W$ is an irreducible component of $D^{=1}_Y$. 
We can construct a dlt blow-up 
$f': Y'\to X$ such that 
$K_{Y'}+D_{Y'}=f'^*(K_X+D)$ and 
that $f'^{-1}\circ f:Y\dashrightarrow Y'$ is an isomorphism at the 
generic point of $W$ (cf.~\cite[Section 6]{fujino-ss}). 
Since $K_{Y'}+D_{Y'}\sim _{\mathbb Q}0$ and we 
can easily check that $\widetilde \mu(Y', D_{Y'})=0$ 
(cf.~\cite[Claim $(A_n)$]{fujino}), 
we see that 
$W'$, the strict transform of $W$, is rationally connected and $h^i(W', \mathcal O_{W'})=0$ 
for every $i>0$ by Proposition \ref{15}. Thus, $W$ is rationally 
connected (cf.~\cite[Corollary 1.5]{hm}) and $h^i(W, \mathcal O_W)=0$ for 
every $i>0$. 
\end{proof}

\section{Isolated log canonical singularities with index one}\label{sec2} 
In this section, we consider when 
an isolated log canonical singularity 
with index one is Cohen--Macaulay or not. 

\begin{say}\label{211}  
Let $P\in X$ be an $n$-dimensional isolated 
lc singularity with index one. 
By the algebraization theorem (cf.~\cite{hironaka-rossi}, 
\cite[Corollary 1.6]{artin1}, and \cite[Theorem 3.8]{artin2}), 
we always assume that 
$X$ is an algebraic variety in this paper 
(see also \cite[Theorems 3.2.3 and 3.2.4]{ishii-book}). 
Assume that $P\in X$ is not lt. 
We consider a resolution 
$f:Y\to X$ 
such that (i) $f$ is an isomorphism outside 
$P\in X$, 
and (ii) $f^{-1}(P)$ is a simple normal crossing 
divisor on $Y$. 
In this setting, we can 
write $$K_Y=f^*K_X+F-E,$$ 
where 
$F$ and $E$ are both effective Cartier 
divisors without common irreducible 
components. In particular, $E$ is a reduced simple normal 
crossing 
divisor on $Y$. 
\end{say}

\begin{lem}\label{42}  
The cohomology group 
$H^i(E, \mathcal O_E)$ is independent of $f$ for every $i$. 
\end{lem}

\begin{proof}
Let $f':Y'\to X$ be another resolution 
with 
$K_{Y'}=f'^*K_X+F'-E'$ as in \ref{211}. 
By the weak factorization theorem (see 
\cite[Theorem 5-4-1]{matsuki} or \cite[Theorem 0.3.1(6)]{akmw}), 
we may assume that $\varphi:Y'\to Y$ is a blow-up 
whose center $C\subset \Supp f^{-1}(P)$ is 
smooth, irreducible, and has simple normal crossing with $\Supp f^{-1}(P)$. 
It means that at each point $p\in \Supp f^{-1}(P)$ there 
exists a regular coordinate system $\{x_1, \cdots, x_n\}$ in a neighborhood 
$p\in U_p$ such that 
$$\Supp f^{-1}(P)\cap U_p=\left\{\prod_{j\in J}x_j=0\right\}$$ and 
$C\cap U_p=\{x_i=0 \ \text{for}\ i\in I\}$ for some 
subsets $I, J\subset \{1, \cdots, n\}$. 
Thus, we can directly check that 
$H^i(E, \mathcal O_E)\simeq 
H^i(E', \mathcal O_{E'})$ for every $i$. 
\end{proof}

\begin{say}
Let $\Gamma$ be the dual 
complex of $E$ and let $|\Gamma|$ be the topological 
realization of $\Gamma$. 
Note that the vertices of $\Gamma$ correspond to the 
components $E_i$, 
the edges correspond to $E_i\cap E_j$, and so on, where 
$E=\sum _i E_i$ is the irreducible 
decomposition of $E$. 
More precisely, $E$ defines a conical 
polyhedral complex $\Delta$ (see \cite[Chapter II, Definition 5]{kkms}). 
By \cite[p.70 Remark]{kkms}, we get a compact polyhedral complex 
$\Delta_0$ from $\Delta$. 
The dual complex $\Gamma$ of $E$ is essentially the same as this compact 
polyhedral complex $\Delta_0$ and $|\Gamma|=|\Delta_0|$ as 
topological spaces. 
See the construction of the {\em{dual complex}} 
in \cite{stepanov} and \cite[Section 2]{payne} for details. 
Therefore, we obtain the following lemma. 
\end{say} 
\begin{lem}
The dual complex $\Gamma$ is well defined and 
$|\Gamma|$ is independent of $f$. 
\end{lem}
\begin{proof}
As we explained above, the well-definedness 
of $\Gamma$ is in \cite[Chapter II]{kkms}. 
By the weak factorization theorem (see \cite[Theorem 5-4-1]{matsuki} 
or \cite[Theorem 0.3.1(6)]{akmw}), 
we can easily check that the topological 
realization $|\Gamma|$ does not depend 
on $f$.  
\end{proof}

\begin{rem}
The paper \cite{stepanov} 
discusses the dual complex of $\Supp f^{-1}(P)$ by the 
same method. 
Case 1) in the proof of \cite[Lemma]{stepanov} 
is sufficient for our purposes. 
Note that we treat the dual complex $\Gamma$ of $E$. 
In general, $\Supp E\subsetneq \Supp f^{-1}(P)$.  
\end{rem}

\begin{say}\label{217} 
Let $g:Z\to X$ be a projective birational 
morphism as in Lemma \ref{conj1}. 
Then we have 
$0\to \mathcal O_Z(-D)\to \mathcal O_Z\to 
\mathcal O_D\to 0$. 
By the vanishing theorem, 
we obtain $R^ig_*\mathcal O_Z(K_Z)=0$ for every $i>0$. 
Therefore, we have 
$$R^ig_*\mathcal O_Z\simeq 
R^ig_*\mathcal O_D\simeq 
H^i(D, \mathcal O_D)$$ for 
every $i>0$. 
We note that $D$ is connected since $\mathcal O_X\simeq 
g_*\mathcal O_Z\to g_*\mathcal O_D$ is surjective. 
By applying Corollary \ref{35}, we can construct a resolution 
$h:Y\to Z$ such that $$K_Y+E-F=h^*(K_Z+D)=f^*K_X, $$
where $F$ and $E$ are both effective Cartier divisors without common 
irreducible 
components, $\Supp E$ is a simple normal crossing divisor, 
$f=g\circ h$, 
$h$ is an isomorphism outside $g^{-1}(P)$, 
$h$ is an isomorphism over the generic point of any lc center 
of $(Z, D)$, $R^ih_*\mathcal O_E=0$ for 
every $i>0$, and $h_*\mathcal O_E\simeq \mathcal O_D$. 
Therefore, $H^i(D, \mathcal O_D)\simeq H^i(E, \mathcal O_E)$ for 
every $i$. 
Apply the principalization to the defining ideal sheaf 
$\mathcal I$ of $f^{-1}(P)$. 
Then we obtain a sequence of blow-ups whose 
centers have simple normal crossing with $E$ (cf.~\cite[Theorem 3.35]{kollar-book}). 
In this process, $H^i(E, \mathcal O_E)$ does not change for every $i$ 
(cf.~the proof of Lemma \ref{42}). 
Therefore, we may assume that $f^{-1}(P)$ is a divisor on $Y$. 
We further take a sequence of blow-ups whose centers have simple normal 
crossing with $E$. 
Then we can make $\Supp f^{-1}(P)$ a simple normal crossing divisor 
on $Y$ 
(cf.~\cite[Corollary 7.9]{bev} or \cite[Proposition 6]{kollar-semi}). 
We note that we may assume that $f$ is an isomorphism outside $P\in X$.  
We also note that 
$R^ig_*\mathcal O_Z\simeq R^if_*\mathcal O_Y$ for every 
$i$ because $Z$ has only rational singularities. 
So, we obtain the next proposition. 
\end{say}

\begin{prop}\label{46} 
Let $f:Y\to X$ be a resolution as in {\em{\ref{211}}}. 
Then $R^if_*\mathcal O_Y\simeq H^i(E, \mathcal O_E)$ 
for every $i>0$. Therefore, 
$P\in X$ is Cohen--Macaulay, 
equivalently, $P\in X$ is Gorenstein, if and only if 
$H^i(E, \mathcal O_E)=0$ for $0<i<n-1$. 
\end{prop}
\begin{proof}
It is a direct consequence of Lemma \ref{42} and 
Corollary \ref{333} by \ref{217}. 
\end{proof}

\begin{rem}\label{4646}
In \ref{217}, $(K_Z+D)|_D=K_D\sim 0$. 
Therefore, $H^{n-1}(D, \mathcal O_D)$ is dual to 
$H^0(D, \mathcal O_D)$, where 
$n=\dim X$. 
So, $R^{n-1}g_*\mathcal O_Z\simeq \mathbb C(P)$. 
Thus, $P\in X$ is not a rational singularity. 
\end{rem}

\begin{rem}
Shihoko Ishii proves $$R^if_*\mathcal O_Y\simeq H^i(f^{-1}(P)_{\mathrm {red}}, 
\mathcal O_{f^{-1}(P)_{\mathrm {red}}})$$ for every $i>0$ by the 
theory of Du Bois singularities 
(cf.~\cite[Corollary 1.5, Theorem 2.3]{ishii} 
and \cite[Proposition 7.1.13, Theorem 7.1.17]{ishii-book}). 
For details, see \cite{ishii} and \cite{ishii-book}. 
\end{rem}

By using the minimal model program with 
scaling, we can prove Proposition \ref{46} without 
appealing to Lemma \ref{42}. 

\begin{rem}\label{rem48}
Let $f:Y\to X$ with $K_Y+E=f^*K_X+F$ be as in \ref{211}. 
Let $H$ be an effective $f$-ample $\mathbb Q$-divisor on $Y$ such 
that $(Y, E+H)$ is dlt and that 
$K_Y+E+H$ is nef over $X$. 
We can run the minimal model program on $K_Y+E$ over $X$ with scaling 
of $H$. 
Then we obtain a dlt blow-up 
$f': Y'\to X$ such that 
$(Y', E')$ is a $\mathbb Q$-factorial dlt pair and that 
$K_{Y'}+E'=f'^*K_X$ where 
$E'$ is the pushforward of $E$ on $Y'$ (cf.~\cite[Section 4]{fujino-ss}). 
We note that each step of the minimal model program 
$$
Y\dashrightarrow Y_1\dashrightarrow Y_2\dashrightarrow \cdots \dashrightarrow 
Y'
$$
is an isomorphism at the generic point of any lc center of $(Y, E)$. 
By \ref{217}, $R^if_*\mathcal O_Y\simeq R^if'_*\mathcal O_{Y'}\simeq 
R^if'_*\mathcal O_{E'}\simeq H^i(E', \mathcal O_{E'})$ for 
every $i>0$. 
By taking a common resolution 
\begin{equation*}
\xymatrix{ & W\ar[dl]_{\alpha} \ar[dr]^{\beta}\\
 Y \ar@{-->}[rr]  & & Y'}
\end{equation*}
such that $\alpha$ (resp.~$\beta$) is an isomorphism over the generic 
point of any lc center of $(Y, E)$ (resp.~$(Y', E')$) 
and that $\Exc (\alpha)$, $\Exc (\beta)$, and $\Exc (\alpha)\cup \Exc (\beta)\cup
\Supp \alpha_*^{-1}E$ are simple normal crossing divisors on $W$, we can easily 
check that 
$$
H^i(E, \mathcal O_{E})\simeq H^i(E', \mathcal O_{E'})
$$ 
for every $i$ because $R\alpha_*\mathcal O_T\simeq 
\mathcal O_E$ and $R\beta_*\mathcal O_{T}\simeq \mathcal O_{E'}$ 
(cf.~Corollary \ref{35}). 
Note that $K_W+\Delta_1=\alpha^*(K_Y+E)$ and $K_W+\Delta_2=\beta^*(K_{Y'}+E')$ 
with $\Delta_1^{=1}=T=\Delta_2^{=1}$ such that $T$ is a reduced simple normal crossing 
divisor on $W$. 
Therefore, 
$$H^i(E, \mathcal O_E)\simeq H^i(E', \mathcal O_{E'})\simeq R^if_*\mathcal O_Y$$ for 
$i>0$. 

Let $E=\sum _i E_i$ be the irreducible decomposition 
and let $E'=\sum _i E'_i$ be the corresponding irreducible decomposition. 
Let $E_{i_0}$ be an irreducible component of $E$ and let $T_{i_0}$ be the 
strict transform of $E_{i_0}$ on $W$. 
By applying the connectedness lemma (cf.~\cite[Theorem 5.48]{km}) 
to $\alpha:T_{i_0}\to E_{i_0}$ and $\beta:T_{i_0}\to E'_{i_0}$, 
we know that the number of the connected components 
of $\sum _{i\ne i_0}E_i|_{E_{i_0}}$ coincides with 
that of $\sum _{i\ne i_0}E'_i|_{E'_{i_0}}$. 
Therefore, $\sum _{i\ne i_0}E_i|_{E_{i_0}}$ has at most two connected components 
by applying Proposition \ref{216} to $(E'_{i_0}, \sum _{i\ne i_0}E'_i|_{E'_{i_0}})$. 
Note that $(E'_{i_0}, \sum _{i\ne i_0}E'_i|_{E'_{i_0}})$ is dlt and 
$K_{E'_{i_0}}+\sum _{i\ne i_0}E'_i|_{E'_{i_0}}\sim 0$. 
\end{rem}

\begin{say}[Invariant $\mu$]\label{4111} 
Let $P\in X$ be an isolated lc singularity with index one which 
is not lt. Let $g:Z\to X$ be a projective 
birational morphism such that 
$K_Z+D=g^*K_X$ and that $(Z, D)$ is a $\mathbb Q$-factorial dlt pair. 
We define 
$$\mu=\mu(P\in X)=\min \{\,\dim W \,| \, 
W \ {\text{is an lc center of}} \ (Z, D)\}. 
$$ 
This invariant $\mu$ was first introduced in \cite[Definition 4.12]{fujino2}. 
Let $D=\sum _i D_i$ be the irreducible decomposition. 
Then $K_{D_i}+\Delta_i:=(K_Z+D)|_{D_i}\sim 0$ and 
$(D_i, \Delta_i)$ is dlt. 
By applying Proposition \ref{216} to 
each $(D_i, \Delta_i)$, every minimal lc center of 
$(Z, D)$ is $\mu$-dimensional and 
all the minimal lc centers are 
birational each other. 
Note that $D$ is connected. 

Let $g':Z'\to X$ be another projective birational morphism 
such that $K_{Z'}+D'=g'^*K_{X}$ and 
that $(Z', D')$ is a $\mathbb Q$-factorial 
dlt pair. 
Then it is easy to see that 
$(Z, D)\dashrightarrow (Z', D')$ is $B$-birational. 
This means that there is a common resolution 
\begin{equation*}
\xymatrix{ & W\ar[dl]_{\alpha} \ar[dr]^{\beta}\\
 Z \ar@{-->}[rr]  & & Z'}
\end{equation*}
such that $\alpha^*(K_Z+D)=\beta^*(K_{Z'}+D')$. 
Then we can easily check that 
\begin{align*}
&\min \{\,\dim W \,| \, 
W \ {\text{is an lc center of}} \ (Z, D)\}
\\ &=\min \{\,\dim W' \,| \, 
W' \ {\text{is an lc center of}} \ (Z', D')\}.
\end{align*}
See, for example, the proof of \cite[Lemma 4.9]{fujino}. 
Therefore, $\mu(P\in X)$ is well-defined. 
Let $f:Y\to X$ with $K_Y=f^*K_X+F-E$ be as in \ref{211}. 
Then it is easy to see that 
$$\mu=\mu(P\in X)=\min \{\,\dim W \,| \, 
W \ {\text{is a stratum of}} \ E\} 
$$ 
by Remark \ref{rem48}. 
\end{say}

Now, the following theorem is 
not difficult to prove. 

\begin{thm}\label{4-10} 
We use the notation in {\em{\ref{211}}}. 
We assume $\mu (P\in X)=0$. 
Then $H^i(E, \mathcal O_E)\simeq 
H^i(|\Gamma|, \mathbb C)$. Therefore, 
$P\in X$ is Cohen--Macaulay, equivalently, 
$P\in X$ is Gorenstein, if and only if 
$$
H^i(|\Gamma|, \mathbb C)=\left\{
\begin{array}{ll}
\mathbb C &\text{for}\  i=0, n-1, \\
0 &\text{otherwise.} 
\end{array}
\right.
$$ 
\end{thm}
\begin{proof}
We use the spectral sequence in \ref{2-11} to calculate 
$H^i(E, \mathcal O_E)$. 
By Corollary \ref{3777}, $H^q(E^{[p]}, \mathcal O_{E^{[p]}})=0$ for 
every $q>0$. 
Therefore, we obtain $E_2^{i, 0}\simeq H^i(|\Gamma|, \mathbb C)$ for every $i$ 
and the spectral sequence degenerates at $E_2$. 
Thus we have $H^i(E, \mathcal O_E)\simeq 
H^i(|\Gamma|, \mathbb C)$ for every $i$.  
\end{proof}

\begin{say}\label{2-11}
Let $E$ be a simple normal crossing variety and let $E=\sum _iE_i$ 
be the irreducible decomposition. 
We put $E^{[0]}=\coprod_i E_i$, $E^{[1]}=\coprod _{i, j} (E_i\cap 
E_j)$, $\cdots$, $E^{[p]}=\coprod _{i_0, \cdots, i_p}(E_{i_0}
\cap \cdots \cap E_{i_p})$, $\cdots$. Let 
$a_p:E^{[p]}\to E$ be the obvious map. Then it is well known that 
$$(a_0)_*\mathcal O_{E^{[0]}}\to 
(a_1)_*\mathcal O_{E^{[1]}}\to \cdots \to 
(a_p)_*\mathcal O_{E^{[p]}}\to \cdots
$$ is a resolution of $\mathcal O_E$. 
By taking the associated 
hypercohomology, we obtain 
a spectral sequence 
$$E^{p,q}_1=H^q(E^{[p]}, \mathcal O_{E^{[p]}})\Rightarrow 
H^{p+q}(E, \mathcal O_E).$$ 
\end{say}

We close this section with the following 
obvious two propositions. 

\begin{prop} 
We assume that the dimension of $X$ is $\geq 3$. 
By the above spectral sequence, if $P\in X$ is 
Cohen--Macaulay, then $H^1(|\Gamma|, \mathbb C)=0$. 
\end{prop}
\begin{proof}
By the spectral sequence in \ref{2-11}, 
it is easy to see that $H^1(|\Gamma|, \mathbb C)\ne 0$ implies 
$H^1(E, \mathcal O_E)\ne 0$. 
\end{proof}

\begin{prop}
Let $P\in X$ be an $n$-dimensional 
isolated lc singularity with index one which is not lt. 
If $P\in X$ is Cohen--Macaulay, then 
\begin{align*}
\chi (\mathcal O_E)&:=\sum _i (-1)^i h^i(E, \mathcal O_E)=
1+(-1)^{n-1}\\ & =\sum_{p,q} (-1)^{p+q} \dim H^q(E^{[p]}, \mathcal O_{E^{[p]}}).
\end{align*}
\end{prop}

\begin{rem}
Tsuchihashi's cusp singularities (cf.~\cite{t1} and \cite{t2}) give 
us many examples of three dimensional index one 
isolated lc singularities with $\mu=0$ which are not Cohen--Macaulay. 
\end{rem}

\section{Ishii's Hodge theoretic invariant}\label{sec5}

In this section, we give a Hodge theoretic characterization 
of our invariant $\mu$. 
It shows that 
our invariant $\mu$ coincides with 
Ishii's Hodge theoretic invariant. 

Let us quickly recall Ishii's definition of singularities of type $(0, i)$. 
For the details, see \cite[Section 7]{ishii-book} and \cite[2.6 and 
Definition 2.7]{ishii-q}. 

\begin{say}[Type $(0, i)$ singularities due to Shihoko Ishii] 
Let $P\in X$ be an $n$-dimensional isolated lc singularity with index one which is not lt. 
Let $f:Y\to X$ be a resolution such that 
$$K_Y=f^*K_X+F-E$$ as in \ref{211}. 
Shihoko 
Ishii proves that $H^{n-1}(E, \mathcal O_E)=\mathbb C$ (cf.~Proposition \ref{46} and 
Remark \ref{4646}).  
In \cite[Definition 7.4.5]{ishii-book} and \cite[Definition 2.7]{ishii-q}, 
she defines that the singularity $P\in X$ is of type $(0, i)$ 
if $$\Gr^W_iH^{n-1}(E, \mathcal O_E)\ne 0.$$ 
Note that 
$E$ is a projective simple normal crossing variety, $W$ is 
the weight filtration of the natural mixed Hodge structure on 
$H^{n-1}(E, \mathbb C)$, and that 
$H^{n-1}(E, \mathcal O_E)\simeq \Gr ^0_FH^{n-1}(E, \mathbb C)$ where 
$F$ is the natural Hodge filtration. Therefore, we have 
\begin{align*}
&\Gr ^W_kH^{n-1}(E, \mathcal O_{E})\\
&\simeq \Gr ^W_k \Gr ^0 _FH^{n-1}(E, \mathbb C)\\ 
&\simeq \Gr ^0_F\Gr ^W _k H^{n-1} (E, \mathbb C) 
\end{align*}
By Deligne's theory of mixed Hodge 
structures, we know that $0\leq i\leq n-1$. 
\end{say}

The main purpose of this section is to show that $\mu(P\in X)=i$ where 
$P\in X$ is of type $(0, i)$. 

The following theorem corresponds 
to \cite[Theorem 4.3]{ishii} in our framework. 
For the definition of {\em{sdlt pairs}}, see \cite[Definition 1.1]{fujino}. 
Let $(X, \Delta)$ be an sdlt pair. Then 
$X$ is $S_2$, normal crossing in codimension one, and every irreducible component of $X$ is 
normal. 
Let $V$ be sdlt. Then there is the smallest Zariski closed subset $Z$ of $V$ such that 
$V\setminus Z$ is a simple normal crossing variety and the codimension 
of $Z$ in $V$ is $\geq 2$. 
We define a {\em{stratum}} of $V$ as the closure of a stratum of $V\setminus Z$.  

\begin{thm}\label{new1} 
Let $V$ be an $m$-dimensional connected projective sdlt variety such that 
$K_V\sim 0$. 
Let $f:V'\to V$ be a projective birational morphism from a simple normal crossing 
variety $V'$. 
Assume that there is a Zariski open subset $U'$ {\em{(}}resp.~$U${\em{)}} 
of $V'$ {\em{(}}resp.~$V${\em{)}} 
such that $U'$ {\em{(}}resp.~$U${\em{)}} contains 
the generic point of any stratum of $V'$ {\em{(}}resp.~$V${\em{)}} 
and that $f$ induces an isomorphism between $U'$ and $U$. 
We further assume that the exceptional locus $\Exc(f)$ is a simple normal crossing 
divisor on $V'$ {\em{(}}cf.~\cite[Definition 2.11]{fujino-book}{\em{)}} and 
that 
$$
K_{V'}=f^*K_V+E
$$ 
where $E$ is effective. 
Then $H^m(V', \mathcal O_{V'})=\mathbb C$. 
Moreover, we obtain that 
\begin{align*}
&\Gr ^0_F\Gr ^W _k H^m (V', \mathbb C)\\ 
&\simeq \Gr ^W_k \Gr ^0 _FH^m(V', \mathbb C)\\ 
&\simeq \Gr ^W_kH^m(V', \mathcal O_{V'})\\
&=\begin{cases}
\mathbb C &\text{if}\quad  k=\mu\\ 
0 &\text{otherwise}
\end{cases}
\end{align*}
where $\mu$ is the dimension of the minimal 
stratum of $V'$. 
Note that $F$ is the Hodge filtration and 
$W$ is the weight filtration of the natural mixed Hodge structure on $H^m(V', \mathbb C)$. 
\end{thm}

\begin{proof}
First we prove that $H^m(V', \mathcal O_{V'})=\mathbb C$. 

\begin{step}
Since $V$ is simple normal crossing in codimension one and $S_2$, 
$V$ is semi-normal. We can easily check that $f$ has connected 
fibers by the assumptions. 
Therefore, we obtain $f_*\mathcal O_{V'}\simeq \mathcal O_V$. 
We note that $E$ is $f$-exceptional by the assumptions. 
Since $E$ is effective, $f$-exceptional, and $V$ satisfies Serre's $S_2$ condition, 
we see that $f_*\mathcal O_{V'}(E)\simeq \mathcal O_V$. 
On the other hand, we obtain 
$R^if_*\mathcal O_{V'}(E)\simeq R^if_*\mathcal O_{V'}(K_{V'})=0$ for every 
$i>0$ (cf.~\cite[Lemma 2.33]{fujino-book}). Therefore, 
we have $Rf_*\mathcal O_{V'}(E)\simeq \mathcal O_V$ in the derived category. 
By the same arguments as in the proof of 
Proposition \ref{31}, 
we obtain that $V$ is Cohen--Macaulay. Moreover, 
$R^if_*\mathcal O_{V'}=0$ for every $i>0$ (see the proof of 
Proposition \ref{31}) and $f_*\mathcal O_{V'}\simeq \mathcal O_V$. 
Thus, $H^m(V', \mathcal O_{V'})\simeq H^m(V, \mathcal O_V)=\mathbb C$. 
We note that $K_V\sim 0$ and $V$ is Cohen--Macaulay. 
\end{step}

We use the induction on dimension for the latter statement. 
The statement is obvious for a $0$-dimensional variety. 
\begin{step}
When $V$ is irreducible, the statement is obvious. 
It is because $V'$ is a smooth connected projective variety. 
So, $H^m(V', \mathbb C)$ has the natural pure Hodge structure 
of weight $m$. 
\end{step}
\begin{step}\label{ste33} 
From now on, we assume that $V$ is reducible. 
Let $V'_1$ be an irreducible component of $V'$ and 
let $V_1$ be the corresponding irreducible component of $V$. 
We write $V'=V'_1\cup V'_2$ and $V=V_1\cup V_2$. 
Consider the Mayer-Vietoris exact sequence: 
\begin{align*}\tag{$\spadesuit$}\label{siki1}
H^{m-1}(V'_1\cap V'_2, \mathcal O_{V'_1\cap V'_2})
&\overset{\delta}\to H^m(V', \mathcal O_{V'}) \\ &\to H^m(V'_1, \mathcal O_{V'_1})\oplus 
H^m(V'_2, \mathcal O_{V'_2}). 
\end{align*}
By the Serre duality, $H^m(V'_i, \mathcal O_{V'_i})$ is 
dual to $H^0(V'_i, \mathcal O_{V'_i}(K_{V'_i}))$. 
We put $f_i=f|_{V'_i}$ for $i=1, 2$. 
We can write 
$$
K_{V'_i}+V'_j|_{V'_i}=f_i^*(K_{V_i}+V_j|_{V_i})+E|_{V'_i}\sim E|_{V'_i}=:F_i
$$ 
for $\{i, j\}=\{1, 2\}$ where $F_i$ is an effective $f_i$-exceptional 
divisor. 
We note that $K_{V_i}+V_j|_{V_i}=K_V|_{V_i}\sim 0$. 
Let $H$ be an ample Cartier divisor on $V$. 
Then $(f_i^*H)^{m-1}\cdot K_{V'_i}<0$ because 
$V'_j|_{V'_i}\ne 0$ for $i=1, 2$. 
Thus $H^0(V'_i, \mathcal O_{V'_i}(K_{V'_i}))=0$ for $i=1, 2$. 
This means that $H^m(V'_i, \mathcal O_{V'_i})=0$ for 
$i=1, 2$. 
So the last term in (\ref{siki1}) is zero. 
Therefore, 
we obtain that 
$$
\Gr ^W_kH^{m-1}(V'_1\cap V'_2, \mathcal O_{V'_1\cap V'_2})\to 
\Gr^W_kH^m(V', \mathcal O_{V'})
$$ 
is surjective for every $k$. 
We note that $V'_1\cap V'_2$ is an $(m-1)$-dimensional projective simple normal crossing 
variety and that $V'_1\cap V'_2$ has at most two connected components 
by Proposition \ref{216} and \cite[Theorem 5.48]{km}. 
Note that $(V_1, V_2|_{V_1})$ is dlt and 
$K_{V_1}+V_2|_{V_1}\sim 0$. 
Moreover, each connected component of $V'_1\cap V'_2$ satisfies the 
assumptions of this theorem and the dimension of the minimal stratum of 
each connected component of $V'_1\cap V'_2$ is also $\mu$. 
Therefore, by the induction on dimension, we obtain that 
$\Gr^W_kH^m(V', \mathcal O_{V'})\ne 0$ if and only if $k=\mu$. 
\end{step}
We obtain all the desired results. 
\end{proof}

\begin{rem}\label{rem-nn} 
By Step \ref{ste33} in the proof of Theorem \ref{new1}, 
we obtain the following description. 
Let $C'$ be any minimal stratum of $V'$. Then we 
obtain an isomorphism 
$$
\mathbb C=H^{\mu}(C', \mathcal O_{C'})
\underset{\delta_{\mu}}\simeq \cdots\underset{\delta_k}\simeq \cdots 
\underset{\delta_{m-1}}\simeq H^m(V', \mathcal O_{V'})=\mathbb C
$$
where each $\delta_k$ is the connecting homomorphism 
of a suitable Mayer--Vietoris exact sequence for $\mu \leq k\leq m-1$. 
Note that $C$ has only canonical singularities with $K_C\sim 0$, where 
$C=f(C')$. 
\end{rem}

\begin{rem}[Semi-stable minimal models for 
varieties with trivial canonical divisor] 
Let $f:X\to Y$ be a projective surjective morphism from a smooth 
quasi-projective 
variety $X$ to a smooth 
quasi-projective curve $Y$. Assume that 
$f$ is smooth over $Y\setminus P$, $K_{f^{-1}(Q)}\sim 0$ for every 
$Q\in Y\setminus P$, and 
$f^*P$ is a reduced simple normal crossing divisor on $X$. 
Then we obtain a relative good minimal model $f': X'\to Y$ of $f:X\to Y$ by 
\cite[Theorem 1.1]{fujino-ss}. 
Then the special fiber $S=f'^*P$ is 
an sdlt variety with $K_S\sim 0$. So, we 
can apply Theorem \ref{new1} to $S$.   
\end{rem}

As an application of Theorem \ref{new1}, we obtain the following theorem. 

\begin{thm}\label{new2} 
Let $P\in X$ be an isolated lc singularity with index one which is not 
lt. 
Then $P\in X$ is of type $(0, i)$ if and only if 
$\mu(P\in X)=i$. 
\end{thm}

\begin{proof}
We use the notations in Remark \ref{rem48}. 
Let $f:Y\to X$ be as in \ref{211}. 
First, we apply Theorem \ref{new1} to 
$\beta:T\to E'$. 
Then we obtain 
$$\Gr^W_\mu H^{n-1}(T, \mathcal O_T)\ne 0$$ where 
$\mu=\mu(P\in X)$. 
Next, we consider $\alpha:T\to E$. Let $C$ be a minimal 
stratum of $E$ and let $C'$ be the corresponding stratum of $T$. 
By Step \ref{ste33} in the proof of Theorem \ref{new1}, Remark \ref{rem-nn}, and 
Lemma \ref{lemA}, we can construct the 
following commutative diagram. 
$$
\begin{CD}
\mathbb C=H^{\mu}(C', \mathcal O_{C'})@>{\delta_1}>> H^{n-1}(T, \mathcal O_T)=\mathbb C\\
@A{\alpha|_{C'}^*}AA @AA{\alpha|_T^*}A\\
\mathbb C=H^{\mu}(C, \mathcal O_{C})@>{\delta_2}>> H^{n-1}(E, \mathcal O_E)=\mathbb C
\end{CD}
$$
Note that $\delta_1$ and $\delta_2$ are 
isomorphisms, which are the compositions of 
the connecting homomorphisms of suitable Mayer--Vietoris 
exact sequences (cf.~Remark \ref{rem-nn}), 
and that $\alpha|_{C'}^*$ and $\alpha|_T^*$ are isomorphisms (cf.~Lemma \ref{lemA}). 
By taking $\Gr^W$, we obtain that 
$$
\Gr ^W_\mu H^{n-1}(E, \mathcal O_E)\ne 0. 
$$ 
This means that $P\in X$ is of type 
$(0, \mu)$. We note that 
$$
\Gr ^W_\mu H^{\mu}(C, \mathcal O_C)=H^{\mu}(C, \mathcal O_C)$$ since 
$C$ is smooth and projective. 
\end{proof}

We note that Theorem \ref{new2} also follows from \cite[Proposition 7.4.8]{ishii-book} and 
\cite{ishii-letter} (see \cite[Remark 4.13]{fujino2}). 

Anyway, by Theorem \ref{new2}, our approach in \cite{fujino2} and 
this paper is compatible with Ishii's theory developed 
in \cite{ishii}, \cite{ishii-book}, and \cite{ishii-q}.  


\end{document}